\documentclass[12pt,a4paper]{amsart}
\usepackage{amssymb}
\usepackage{amsmath}
\usepackage{stmaryrd}
\usepackage[latin1]{inputenc}

\begin{document}
\def\i{\,\lrcorner\,}
\def\a{\alpha}
\def\b{\beta}
\def\c{\gamma}
\def\G{\Gamma}
\def\ni{\noindent}
\def\vs{\vskip .6cm}
\def\ss{\smallskip}
\def\la{\langle}
\def\ra{\rangle}
\def\.{\cdot}
\def\O{\Omega}
\def\n{\nabla}
\def\l{\lambda}
\def\t{\tilde}
\def\beq{\begin{equation}}
\def\eeq{\end{equation}}
\def\bea{\begin{eqnarray*}}
\def\eea{\end{eqnarray*}}
\def\ba{\begin{array}}
\def\ea{\end{array}}
\def\x{\times}
\def\f{\varphi}
\def\ms{\medskip}
\def\o{\omega}
\def\ld{\ldots}
\def\e{\varepsilon}
\def\L{\Lambda}
\def\k{\kappa}
\def\r{\end{proof}}
\def\res{\arrowvert}
\def\I{{\mathcal I}}
\def\pa{\partial}
\def\dt{\pa_t}
\def\ds{\pa_s}
\def\ci{{\mathcal C}^\infty}
\def\CP{\CM{\rm P}}

\def \RM{\mathbb{R}}
\def \QM{\mathbb{Q}}
\def \N{\mathbb{N}}
\def \ZM{\mathbb{Z}}
\def \CM{\mathbb{C}}
\def \TM{\mathbb{T}}
\def \HM{\mathbb{H}}
\def\End{{\rm End}}


\def\d{{\delta}}


\def\es{\,\lrcorner\,}
\def\f{\varphi}
\def\Ric{\mathrm{Ric}}
\def\grad{\mathrm{grad}}
\def\id{\mathrm{id}}
\def\Lie{{\mathcal L}}
\def\U{\rm{U}}
\def\SU{\rm{SU}}
\def\OO{\rm{O}}
\def\tr{\mathrm{tr}}
\def\vol{\mathrm{vol}}
\def\hol{\mathfrak{hol}}
\def\so{\mathfrak{so}}
\def\R{\mathbb{R}}
\def\SS{{\mathfrak S}}
\def\Ker{\mathrm{Ker}}
\def\bp{\begin{proof}}
\def\LL{\mathcal {L}}
\def\ra{\rightarrow}
\def\lra{\longrightarrow}
\def\endproof{\hfill$\square$}


\newtheorem{ede}{Definition}[section]

\newtheorem{epr}[ede]{Proposition}

\newtheorem{ath}[ede]{Theorem}

\newtheorem{elem}[ede]{Lemma}

\newtheorem{ere}[ede]{Remark}

\newtheorem{ecor}[ede]{Corollary}

\newtheorem{eex}[ede]{Example}
 
                             
\title{Essential points of conformal vector fields}
\author{Florin Belgun}
\address{Institut f\"ur Mathematik, Humboldt-Universit\"at zu Berlin,
  10099 Berlin, Germany}
\email{belgun@math.hu-berlin.de}
\author{Andrei Moroianu}
\address{Centre de Math{\'e}matiques, Ecole Polytechnique, 91128
  Palaiseau Cedex, France} 
\email{am@math.polytechnique.fr}
\author{Liviu Ornea}
\address{Univ. of Bucharest, Faculty of Mathematics,
14 Academiei str., 70109 Bucharest, Romania, and 
Institute of Mathematics ``Simion Stoilow" of the Romanian Academy,
21, Calea Grivitei str., 
010702-Bucharest, Romania.}
\email{lornea@gta.math.unibuc.ro, Liviu.Ornea@imar.ro}
\thanks{Part of this work was accomplished in the framework of the
  Associated European Laboratory ``MathMode". L.O. is partially supported by
CNCSIS PNII  grant code 8, 525/2009. F.B. is partially supported by
the SFB 647 of the DFG}
\begin{abstract}
An essential point of a conformal vector field $\xi$ on a conformal
manifold $(M,c)$
is a point around which the local flow of $\xi$ preserves no metric in
the conformal class $c$. It is well-known
that a conformal vector field vanishes at each essential point. In
this note we show that essential points are
isolated. This is a generalization to higher dimensions of the
fact that the zeros of a holomorphic function are isolated. 
As an application, we show that every
connected component of the zero set of a conformal vector field is
totally umbilical.  

\vs


\end{abstract}

\maketitle

\section{Introduction}

It is a classical result by D. Alekseevskii \cite{a} that the group of
conformal automorphisms of 
a Riemannian manifold either fixes a conformally equivalent metric or
the manifold is conformally flat. A gap in his proof, found by
R.J. Zimmer and K.R. Gutschera in 1992, was filled by J. Ferrand \cite{f}.

From an infinitesimal point
of view, if a conformal manifold admits a complete 
and {\em essential} conformal vector field (whose global
flow acts by conformal transformations, but not by isometries with
respect to any compatible metric), then it is conformally flat,
\cite{kr}. More recently, Ch. Frances proved a local version of this result
\cite{fr}, see also \cite{fm} and \cite {l} for the pseudo-Riemannian setting. 

For a given conformal vector field $\xi$ (whose flow is 
not globally defined in general), there might exist local metrics in
the conformal class preserved by the local flow of
$\xi$. The union of the definition domains of these $\xi$-invariant
local metrics is the open set of {\em non-essential} points of
the conformal vector field $\xi$. This motivates the following:

\begin{ede} Let $(M,c)$ be a conformal manifold and let $\xi$ be a conformal
  vector field on $M$. A point $x\in M$ is called {\em essential} for
  $\xi$ if there is no local metric in $c$ preserved by the 
  local flow of $\xi$ around $x$.
\end{ede}

If $\xi$ does not vanish at some point $p$, it is easy to see that $p$
is not essential. Indeed, $\xi$ is rectifiable in a neighborhood $U$ of
$p$, so there exists a system of coordinates $(x_1,\ldots,x_n)$ near
$p$ such that $x_i(p)=0$ for all $i$ and $\xi=\partial/\partial x_1$. Take any
metric in $c$, restrict it to the hypersurface
$U\cap \{x_1=0\}$ and extend it to $U$ as to be constant with respect
to $x_1$. Then $\xi$ preserves the new metric, which moreover belongs
to the conformal class $c$ because $\xi$ is conformal. 

In this note, we prove 

\begin{ath}\label{ess} The essential set of a conformal vector field $\xi$ on a
  Riemannian manifold $M$ of dimension $n\ge 2$ consists of isolated
  zeros of $\xi$. 
\end{ath}

Note that this is equivalent for $n=2$ to the fact that the 
zeros of a non-constant holomorphic function are isolated. 
The theorem above may thus be seen as an extension to higher dimensions of this
classical fact. 

For the proof, we focus on the zero set of $\xi$ and show that if a
zero point $x$ of $\xi$ is not isolated then an algebraic condition
(Theorem \ref{cap}) is satisfied by the derivative of $\xi$ at $x$, which
implies, via a result by Beig \cite{beig} and Capocci \cite{C},
that $\xi$ is Killing with respect to a local metric around $x$.

As an application, we show in Theorem \ref{umb} that every connected
component of the zero set of $\xi$ is a totally umbilical submanifold,
thus generalizing a well-known result of S. Kobayashi \cite{kob},
that states that the connected components of the 
set of zeros of a Killing vector field on a Riemannian manifold 
are totally geodesic submanifolds of even codimension. 

A proof of Theorem \ref{umb} was previously given by
D.E. Blair  \cite{bla} under the additional restriction (required by
his proof based on the Obata theorem) that the manifold is
compact. Instead, our proof is purely local, as is the proof of the
above mentioned Kobayashi's result. 

Very recently, results similar to Theorem \ref{ess} were obtained
independently by M. Lampe in his PhD Thesis \cite{l}.

\section{Proof of the main result}
Let $(M^n,c)$ ($n\ge 3$) be a conformal manifold. We choose a
background Riemannian 
metric $g\in c$ and make the usual identifications between 
vectors and 1-forms, 2-forms and skew-symmetric endomorphisms etc.,
induced by this metric.

A vector field $\xi$ is called {\em conformal} if the symmetric part of its covariant
derivative with respect to the Levi-Civita connection of $g$ 
is reduced to its trace:
\beq\label{1} \n_Y\xi=\frac12 Y\i d\xi+\f Y,\qquad \forall Y\in TM.
\eeq  
The function $\f$ is equal to $-\delta^g\xi/n$, but we will not need this in the sequel.

Let $Z_{\xi}$ and $E_\xi$ denote the zero set and the essential set of $\xi$ respectively.
We have seen that $E_\xi\subset Z_\xi$, so in order to prove Theorem \ref{ess},
it will be enough to show that every limit
point\footnote{A {\em limit point} of a set $S$ in a topological space $X$ is a point $x\in S$ such that every neighborhood of $x$ intersects $S\setminus \{x\}$. The set of limit points of $S$ is denoted by $S'$.} $x$ of $Z_{\xi}$ is not essential,
{\em i.e.} there exists a $\xi$-invariant metric defined locally around $x$.

To do that, we make use of the following criterion by Capocci \cite[Theorem 2.1]{C}, (cf. also
\cite{beig}):

\begin{ath} \cite{C}\label{cap} Let $x\in Z_\xi$ be a zero of the conformal
  field $\xi$ on a Riemannian manifold $(M,g)$ of dimension at least $3$ and let $\f$ be the function defined by \eqref{1}. Then $\xi$ is a homothetic vector field with
  respect to some conformal metric $\tilde g$ defined on a
  neighborhood of $x$ if and only if the gradient of 
$\f$ with respect to $g$ belongs to the image of $\n \xi$ at
$x$; moreover, $\xi$ is Killing with respect to $\tilde g$ if, in
addition, $\f(x)=0$. 
\end{ath}

Theorem \ref{ess} will follow directly from Theorem \ref{cap}, together with: 

\begin{ath}\label{fi} Let $x\in Z_{\xi}'$ be a limit point of the 
zero set $Z_\xi$ of a conformal vector field $\xi$. 
Then the function $\f$ defined in \eqref{1} vanishes at $x$ and its gradient 
with respect to $g$ belongs to the image of $\n \xi$ at $x$.
\end{ath}

\begin{proof} Let $x_k\ne x$ be a sequence of zeros of $\xi$ converging to $x$. In a geodesic chart around $x$, we connect $x$ with 
  $x_k$ by uniquely defined minimizing geodesics, denoted by
  $c_k:[0,t_k]\ra M$, $c_k(0)=x,\ c_k(t_k)=x_k,\ \|\dot{c_k}\|=1$.

The function $f_k:=g(\xi,\dot{c_k})$ admits the following
Taylor-Lagrange expansion:
\beq\label{elf}f_k(t_k)=f_k(0)+t_kf'_k(\tau_k),\ \mbox{for } \tau_k\in
(0,t_k).\eeq 
Since $f_k(t_k)=f_k(0)=0$ and $t_k\ne 0$, we have found a sequence of points
$c_k(\tau_k)$, converging to $x$, such that $f'_k(\tau_k)=0$.

On the other hand, equation (\ref{1}) implies that
$f'_k(t)=\f(c_k(t))$, therefore $\f$ vanishes on a sequence converging
to $x$, thus $\f(x)=0$.

For the second part of the theorem, we need to show that $d\f_x$ lies in the image of the
skew-symmetric endomorphism $d\xi_x$ of $T_xM$ for every $x\in Z_{\xi}'$.

For a geodesic $c:[0,T]\ra M$ with $c(0)=x$ and $\|\dot{c}\|=1$, we denote by
$\xi'(t)$ and $\xi''(t)$ the derivatives $\n_{\dot{c(t)}}\xi$,
respectively $\n_{\dot{c(t)}}\n_{\dot{c(t)}}\xi$. From \eqref{1} we have
\beq\label{xip}\xi'(t)=\frac 12
d\xi(\dot{c}(t))+\f(c(t))\dot{c}(t),\eeq
in particular $\xi'(0)=\frac 12 d\xi(\dot{c}(0))$.

\begin{elem}\label{dxi} For a conformal vector field $\xi$ on a
  Riemannian manifold $(M,g)$, such that $\mathcal{L}_\xi g=2\f g$, the
  following relation holds:
$$\n_X d\xi= 2R_{X,\xi}+2d\f\wedge X,\qquad \forall X\in TM.$$\end{elem}

\begin{proof} The notations are those from (\ref{1}), which we will use
in the following equivalent form
\beq\label{dxi2}g(\n_A\xi,B)+g(\n_B\xi,A)=2\f g(A,B),\qquad \forall A,B\in
TM.\eeq

Let $x\in M$ and let $X,Y,Z$ be vector fields parallel at $x$. Since
$d\xi(Y,Z)=g(\n_Y\xi,Z)-g(\n_Z\xi,Y)$, we have at $x$:
\bea(\n_X
d\xi)(Y,Z)&=&\n_X\left(d\xi(Y,Z)\right)=g(\n_X\n_Y\xi,Z)-g(\n_X\n_Z\xi,Y)\\
&=&g(R_{X,Y}\xi,Z)-g(R_{X,Z}\xi,Y)+g(\n_Y\n_X\xi,Z)-g(\n_Z\n_X\xi,Y).\eea
We use (\ref{dxi2}) to substitute the last two terms, and the Bianchi
identity for the first two and get
\begin{align*}(\n_X d\xi)(Y,Z)\ \ =\ \
  \,&g(R_{X,\xi}Y,Z)-g(\n_Y\n_Z\xi,X)+\n_Y(2\f 
g(X,Z))\\ &+g(\n_Z\n_Y\xi,X)-\n_Z(2\f g(Y,X))\\
=\ \ \, &2g(R_{X,\xi}Y,Z)+2d\f (Y)g(X,Z)-2d\f (Z)g(X,Y).
\end{align*}
\end{proof}

We compute then (for clarity, we omit the argument $t$)
\beq\label{xisec}\xi''=R_{\dot{c},\xi}\dot{c}+
(d\f\wedge\dot{c})(\dot{c})+d\f(\dot{c})\dot{c},\eeq 
in particular, since $\xi_x=0$, we have
\beq\label{xis0}\xi''(0)=2d\f_x(\dot{c}(0))\dot{c}(0)-d\f_x.\eeq

We come now to the core of our argument. For an arbitrary geodesic $c$
generated by a unit vector in $T_xM$, we estimate the function
$f(t):=g(\xi_{c(t)},\dot{c}(t))$ using the Taylor expansion of order 2:
\beq\label{ef}f(t)=f(0)+tf'(0)+\frac{t^2}{2}f''(0)+O(t^3).\eeq
Here the function $f$ depends on the chosen geodesic, and the error
term $O(t^3)$ is locally bounded (around $x$) by $t^3K$, where $K$ is
a positive constant depending only on the derivatives of $\xi$ but not
on the geodesic $c$. Note
that $\xi_x=0$ implies $f(0)=0$, and 
that $f'(0)=\f(x)=0$ for any
geodesic $c$.

We choose $c:=c_k$ and $t:=t_k$, and denote by
$f_k(t):=g(\xi_{c_k(t)},\dot{c}(t))$. From \eqref{ef} we infer 
\beq f_k''(0)\lra 0 \mbox{ for } k\lra\infty,\eeq

Restricting $c_k$ to a subsequence for which $\dot{c}_k(0)$
converges to a unit vector $V\in T_xM$ yields
\beq\label{dfv}d\f_x(V)=0.\eeq
\smallskip

In the next step
we estimate $\xi(t):=\xi_{c(t)}$
using a version of the Taylor expansion for vector-valued functions
(here we consider $\xi(.)$ as a function on an interval with values in
$\R^n$, whose components are the components of $\xi$ with respect to
a chosen orthonormal basis parallel along $c$):
\beq\label{tay}\xi(t)=\xi(0)+t\xi'(0)+\frac{t^2}{2}\xi''(0)+O(t^3).\eeq

Note that the function $t\mapsto\xi(t)$ depends on the chosen geodesic $c$,
but the norm of the error term $O(t^3)$ is bounded by $M t^3$ for a
constant $M$ 
depending only on the derivatives of $\xi$. We now set $c:=c_k$ and
$t=t_k$ and denote the corresponding 
function $\xi_{c_k(.)}$ by $\xi_k$. We get
$$0=t_k\xi'_k(0)+\frac{t^2_k}{2}\xi_k''(0)+O(t_k^3),$$
which implies, after taking the quotient by $t_k^2$ and using
(\ref{xip}) and (\ref{xis0}): 
\beq\label{est}\left\|\frac 12 d\xi\left(\frac{\dot{c}_k(0)}{t_k}\right)+
2d\f_x(\dot{c}_k(0))\dot{c}_k(0)-d\f_x \right\|\le M t_k\ \forall
k\in\N.\eeq

Denote now by $V_k$ the vector $\dot{c}_k(0)/2t_k$. The sequence $\{V_k\}$ is
unbounded, but in our relation (\ref{est}) only counts the
projection of $V_k$ on the image of $d\xi_x$, denoted by $W_k:=\pi
V_k$. 

Since $\dot{c}_k(0)\lra V\in T_xM$ and $d\f_x(V)=0$ by \eqref{dfv}, we
conclude that the middle term in (\ref{est}) tends to zero, thus
\beq\label{last}\|d\xi_x(W_k)-d\f_x\|\lra 0,\mbox{ for }
k\lra\infty.\eeq
 Since $W_k\perp\ker(d\xi_x)$, it follows that the norm of $W_k$ is
 bounded, and we can restrict to a subsequence that converges to
 $W\in\mbox{Im}(d\xi_x)$ (Here we consider the kernel and the image
 of $d\xi_x$ as the kernel and the image of a skew-symmetric
 endomorphism of $T_xM$). It thus follows 
$$d\f_x=d\xi_x(W),$$
which finishes the proof.

\end{proof}

{\noindent\bf{Remark. }} We have shown that all zeros of $\xi$ which
are not isolated are not essential. Nothing can be said, in general,
about isolated zeros of a conformal vector field, as the following
classical example shows:

Let $\xi$ be the conformal vector field on the round sphere
$S^n\subset \R^{n+1}$, which is
sent through the stereographic projection from one pole $P$ to a
translation vector field on $\R^n$. Then both $\xi$ and its derivative
vanish at $P$, which is the only zero of $\xi$. Therefore, in the
notations of the equation (\ref{1})), $d\xi_P=0$ and $\f_P=0$. On the
other hand, $d\f_P\ne 0$ (which is implied by the very fact that
$\xi$ is not trivial), so the condition in Theorem \ref{cap} is
violated, which is a proof that $P$ is an essential point for $\xi$.

\section{The zero set of a conformal vector field}

As an application to our main result, we prove 
\begin{ath}\label{umb} Let $\xi$ be a conformal vector field on a
  Riemannian manifold $(M,g)$ of dimension at least $2$,
and let $Z_{\xi}$ be the zero set of $\xi$ on $M$. Then $Z_{\xi}$ is a
disjoint union of embedded connected totally umbilical submanifolds of $M$, of even codimension when not reduced to a point.
\end{ath}
Let $N\subset M$ be an isometrically embedded  submanifold and denote with the same letter $g$ both the metric of $M$ and the induced metric on $N$. Let then $\n^M$, $\n^N$ be the associated Levi-Civita connections on $M$ and $N$. The $(0,2)$ tensor field  $B:T^*N\otimes T^*N\ra \nu(N)$, defined by 
$B(X,Y):=\n^M_XY-\n^N_XY$, where $\nu(N)$ the normal bundle of $N$ in $M$, is called the {\em second fundamental form} of the embedding. 

By definition, the submanifold $N$ is called \emph{totally umbilical} if its 
second fundamental form  is proportional to the metric tensor induced on $N$ in the following sense:
$$B(X,Y)=Hg(X,Y),\ \forall X,Y\in \mathcal{X}(N).$$
The normal vector field $H$ is called the {\em mean curvature} vector field. One-dimensional submanifolds and totally geodesic submanifolds ({\em i.e.} with vanishing second fundamental form) are totally umbilical. 

By definition, every
point is considered to be a totally umbilical submanifold of dimension 0.

Using the relation between the Levi-Civita connections $\n$, $\n'$ of two conformal metrics $g'=e^{2f}g$:
$$\n'=\n+df\otimes Id +Id\otimes df -g\cdot \grad_gf,$$
one easily sees that if $N$ is totally umbilical with respect to a metric
then 
it is totally umbilical with respect to any other conformally equivalent metric. In
particular, $N$ is totally umbilical if there exists a metric $\tilde g$ in the
conformal class, for which $N$ is totally geodesic. 

\begin{proof}[Proof of Theorem \ref{umb}]
First, we apply Theorem \ref{ess} to decompose the zero set $Z_{\xi}$ as the
disjoint union of its isolated points $Z^{iso}$ and non-isolated
points $Z':=Z_{\xi}'$, which are inessential, and 
conclude that every $x\in Z'$ admits a metric $g_x\in [g]$, defined
locally around $x$, such that $\xi$ is Killing with respect to $g_x$, and
$\xi_x=0$. 

The classical result of Kobayashi \cite{kob} implies then that 
there is a neighborhood $U_x$ of $x$ such that $Z_{\xi}\cap
U_x=\exp_x(\ker d\xi_x)\cap U_x=\exp_y(\ker d\xi_y)\cap U_x$, for any
$y\in Z_{\xi}\cap U_x$ (here, we consider the dual 1-form
to $\xi$ with respect to $g_x$, and compute its exterior
derivative $d\xi$ at $x$; Also, the exponential map $\exp_x$ is taken
with respect to $g_x$). 

In conclusion, every point in $Z_{\xi}$ admits a
neighborhood $U_x$ such that $Z_{\xi}\cap U_x$ is a totally umbilical
submanifold of even codimension. By usual connectedness arguments,
these local totally umbilical submanifolds have a common dimension and
glue together to a global submanifold.

\end{proof}

Note that  a  classification of totally umbilical submanifolds exists
only for space forms and, more generally, for locally symmetric spaces
(see {\em e.g.} \cite{che}).  

The above result can then be a useful tool for producing examples of totally umbilical submanifolds or obstructions to the existence of conformal vector fields on certain Riemannian manifolds.


\begin{thebibliography}{22}
\bibitem{a} D.V. Alekseevskii, {\em Groups of conformal transformations
    of Riemannian spaces} (Russian) Mat. Sb. (N.S.) 89(131) (1972),
  280--296 [translated in English in Math. USSR-Sb. {\bf 18}
(1972), 285--301].

\bibitem{beig} R. Beig, {\em Conformal Killing Vectors Near a Fixed
    Point}, Preprint of the Institut f\"ur Theoretische Physik,
  Vienna, 1992.

\bibitem{bla} D.E. Blair, {\em On the zeros of a conformal vector
    field}, Nagoya Math. J. {\bf 55} (1974), 1--3. 

\bibitem{C} M.S. Capocci, {\em Essential conformal vector fields},
  Class. Quantum Grav. {\bf 16} (1999), 927--935. 
  
  \bibitem{che} B.-Y. Chen, Geometry of submanifolds and its
    applications. Science University of Tokyo, Tokyo, 1981.

\bibitem{f} J. Ferrand, {\em The action of conformal transformations
    on a Riemannian manifold}, Math. Ann. {\bf 304}, no. 1 (1996),
  277--291. 

\bibitem{fr} Ch. Frances, {\em Local dynamics of conformal vector
    fields}, arXiv:0909.0044.

\bibitem{fm} Ch. Frances, K. Melnik, {\em Formes normales pour les champs
    conformes pseudo-riemanniens}, arXiv:1008.3781.

\bibitem{kob} S. Kobayashi, {\em Fixed points of isometries}, Nagoya
  Math. J. {\bf 13} (1958) 63--68. 

\bibitem{kr} W. K\"uhnel, H.-B. Rademacher, {\em Conformal transformations of
 pseudo-Riemannian manifolds.}  Recent developments in pseudo-Riemannian
 geometry, ESI Lect. Math. Phys., Eur. Math. Soc., Z\"urich, 2008, 261--298.

\bibitem{l} M. Lampe, {\em On conformal connections and infinitesimal
    conformal transformations}, PhD Thesis, 2010.


\end{thebibliography}
\end{document}